\documentclass[12pt,reqno,fleqn]{amsart}  
\usepackage{tikz}
\usepackage{amsmath,amstext,amsthm,amssymb,amsxtra}
\usepackage{mathrsfs}
\usepackage[top=1.5in, bottom=1.5in, left=1.25in, right=1.25in]	{geometry}
\usepackage{lmodern}

\usepackage{graphics}
\usepackage{euscript}
\usepackage{xcolor}

\usepackage{mathtools}
\mathtoolsset{showonlyrefs,showmanualtags}

\usepackage{hyperref} 
\hypersetup{
	colorlinks=true,       
	linkcolor=blue,          
	citecolor=magenta,        
	filecolor=magenta,      
	urlcolor=cyan           
}

\usepackage[msc-links]{amsrefs}

\newtheorem{theorem}{Theorem}[section]

\newtheorem{proposition}{Proposition}[section]
\newtheorem{lemma}{Lemma}[section]
\newtheorem{corollary}{Corollary}[section]

\theoremstyle{definition}
\newtheorem{definition}{Definition}[section]
\theoremstyle{definition}

\newtheorem{remark}{Remark}[section]

\theoremstyle{remark}

\def\ZF{\ensuremath{\mathscr F}}

\def\BMO{{\rm BMO }}
\def\BLO{{\rm BLO }}

\def\OSC{{\rm OSC}}
\def\LOSC{{\rm LOSC}}
\def\SUP{{\rm SUP}}
\def\INF{{\rm INF}}

\def\esssup{{\rm esssup\, }}
\def\essinf{{\rm essinf\, }}

\def\MM^d{\ensuremath{\mathfrak M}}
\def\MM{\ensuremath{\mathcal M}}
\def\ZM{\ensuremath{\mathfrak M}}
\def\ZB{\ensuremath{\mathscr B}}
\def\ZA{\ensuremath{\mathscr A}}

\def\ZZ{\ensuremath{\mathbb Z}}

\def\ZI{\ensuremath{\textbf 1}}

\def\ZK{\ensuremath{\mathcal K}}

\def\ZR{\ensuremath{\mathbb R}}

\def\ZT{\ensuremath{\mathbb T}}
\def\pr{\ensuremath{\mathrm {pr}}}

\def\ZG{{\mathscr G\,}}

\numberwithin{equation}{section}
\def\md#1#2\emd{\ifx0#1
	\begin{equation*} #2 \end{equation*}\fi  
	\ifx1#1\begin{equation}#2\end{equation}\fi   
	\ifx2#1\begin{align*}#2\end{align*}\fi   
	\ifx3#1\begin{align}#2\end{align}\fi    
	\ifx4#1\begin{gather*}#2\end{gather*}\fi  
	\ifx5#1\begin{gather}#2\end{gather}\fi   
	\ifx6#1\begin{multline*}#2\end{multline*}\fi  
	\ifx7#1\begin{multline}#2\end{multline}\fi  
	\ifx8#1\begin{multline*}\begin{split}#2\end{split}\end{multline*}\fi
	\ifx9#1\begin{multline}\begin{split}#2\end{split}\end{multline}\fi
}
\newcommand {\e }[1]{\eqref{#1}}
\newcommand {\lem }[1]{Lemma \ref{#1}}

\newcommand {\cor }[1]{Corollary \ref{#1}}
\newcommand {\pro }[1]{Proposition \ref{#1}}

\newcommand {\trm }[1]{Theorem \ref{#1}}

\newcommand {\sect }[1]{Section \ref{#1}}
\newcommand {\df }[1]{Definition \ref{#1}}

\title[] {Maximal operators on spaces $\BMO$ and $\BLO$}
\author{Grigori A. Karagulyan}
\address{Institute of Mathematics of NAS of RA, Marshal Baghramian ave., 24/5, Yerevan, 0019, Armenia} 
\address{Faculty of Mathematics and Mechanics, Yerevan State
University, Alex Manoogian, 1, 0025, Yerevan, Armenia} 
\email{g.karagulyan@ysu.am}

\thanks{The work was supported by the Science Committee of RA, in the frames of the research project 21AG‐1A045}

	
\begin{document}

\begin{abstract}
We consider maximal kernel-operators on abstract measure spaces $(X,\mu)$ equipped with a ball-basis. We prove that under certain asymptotic condition on the kernels those operators maps boundedly $\BMO(X)$ into $\BLO(X)$, generalizing the well-known results of Bennett-DeVore-Sharpley \cite{BDS} and Bennett \cite{Ben} for the Hardy-Littlewood maximal function. As a particular case of such an operator one can consider the maximal function 
\begin{equation}
	\MM_\phi f(x)=\sup_{r>0}\frac{1}{r^d}\int_{\ZR^d}|f(t)|\phi\left(\frac{x-t}{r}\right)dt,
\end{equation} 
and its non-tangential version.  Here $\phi(x)\ge 0$ is a bounded spherical function on $\ZR^d$,  decreasing with respect to $|x|$ and satisfying the bound
\begin{equation*}
	\int_{\ZR^d}\phi (x)\log (2+|x|)dx<\infty. 
\end{equation*}
We prove that if $f\in\BMO(\ZR^d)$ and $\MM_\phi(f)$ is not identically infinite, then $\MM_\phi(f)\in \BLO(\ZR^d)$.
Our main result is an inequality, providing an estimation of certain local oscillation of the maximal function $\MM(f)$ by a local sharp function of $f$.  
\end{abstract}

	\maketitle  

\section{Introduction}\label{S4}
\subsection{$\BMO$ and the Hardy-Littlewood maximal function}
The space of functions of bounded mean oscillation, $\BMO(\ZR^d)$, was introduced by John and Nirenberg \cite{JoNi}. This space is the collection of locally integrable functions  $f$ on $\ZR^d$ such that
\begin{equation}\label{x0}
	\|f\|_{\BMO}=\sup_{B}\frac{1}{|B|}\int_B|f-f_B|<\infty,
\end{equation}
where $f_B$ denotes the mean of $f$ over a ball $B\subset \ZR^d$, and the supremum is taken over all the balls $B$. This is not properly a norm, since any function which is constant almost everywhere has zero oscillation. However, modulo constant, $\BMO$ becomes a Banach space with respect to the norm \e{x0}. The space $\BMO$ plays an important role in the harmonic analysis and in the study of partial differential equations. One can check that
$L^\infty(\ZR^d)\subset \BMO(\ZR^d)$, but there exist unbounded $\BMO$ functions. However, the fundamental inequality of John-Nirenberg \cite{JoNi} states that $\BMO$ functions have exponential oscillations over the balls. Namely, for any ball $B\subset \ZR^d$ we have
\begin{equation}\label{JN}
	|\{x\in B:\, |f(x)-f_B|>\lambda\}|\le c_1|B|\exp(-c_2\lambda/\|f\|_\BMO),\quad \lambda>0,
\end{equation}
where $c_1$ and $c_2$ are positive absolute constants. A remarkable results about $\BMO$ are the duality relationship between $\BMO$ and the Hardy space $H^1$ due to Ch.~Fefferman \cite{Fef} and a deep connection between the Carleson measures and the $\BMO$-functions established by Ch.~Fefferman and E.~Stein \cite{FeSt}.
There are different characterizations and generalizations of spaces $\BMO$. Coifman and Rochberg in \cite{CoRo} have introduced the space $\BLO$ of functions of bounded lower oscillation. This space is defined analogously to $\BMO$ except that in \e{x0} one subtracts from $f$ the essential infimum of $f$ on $B$ instead of the average $f_B$. One can easily check that $L^\infty\subset \BLO\subset \BMO$ and moreover,
\begin{equation*}
	\|f\|_\BMO\le 2\|f\|_\BLO,\quad \|f\|_\BLO\le 2\|f\|_\infty.
\end{equation*} 
It was shown in \cite{CoRo} that every $\BMO$-function is representable as a difference of two $\BLO$-functions. 

It is of interest the problem of boundedness of certain operators in harmonic analysis on spaces $\BMO$ and $\BLO$.  The fact that ordinary Calderón–Zyg\-mund operators map $L^\infty$ boundedly into $\BMO$ were independently obtained by Peetre \cite{Pee}, Spanne \cite{Spa} and Stein \cite {Ste}. Peetre \cite{Pee} also observed that the translation-invariant Calderón–Zygmund operators actually map $\BMO$ to itself. We refer also papers  \cite{Lec1, Lec2, Ler, ChFr} for extensions and 
alternative proofs  of these results, including consideration of some operators on spaces $\BLO$.

In the present paper we are mostly interested in boundedness of maximal operators on spaces $\BMO$ and $\BLO$. A classical case of such an operator is the Hardy-Littlewood maximal function, which for a locally integrable function $f$ on $\ZR^d$ is given by
\begin{equation}\label{HL}
	Mf(x)=\sup_{B\ni x}\frac{1}{|B|}\int_B|f|,
\end{equation}
where the supremum is taken over all the balls $B\subset \ZR^d$, containing the point $x\in \ZR^d$. A classical result of Bennett-DeVore-Sharpley \cite{BDS} states that the operator $M$ is bounded on $\BMO(\ZR^d)$. Bennett \cite{Ben} extended this result, proving that the maximal function maps  boundedly $\BMO(\ZR^d)$ into $\BLO(\ZR^d)$. Further works in this direction can be found in \cite{DaGi, ChFr, Ou, Saa}.

In this paper we consider maximal operators of general type. Let $\phi(x)\ge 0$ be a bounded integrable spherical function on $\ZR^d$,  decreasing with respect to $|x|$. Consider the convolution type maximal operator
\begin{equation}\label{x55}
	\MM_\phi f(x)=\sup_{r>0}\frac{1}{r^d}\int_{\ZR^d}|f(t)|\phi\left(\frac{x-t}{r}\right)dt.
\end{equation}
As a corollary of the main theorem of the paper (\trm{T2}), we particularly obtain the following result.
\begin{theorem}\label{T1}
Let $\phi$ be a bounded spherical function on $\ZR^d$,  decreasing with respect to $|x|$ and satisfying the bound
\begin{equation}\label{x56}
I(\phi)	=\int_{\ZR^d}\phi(x)\log (2+|x|)dx<\infty.
\end{equation}
If $f\in \BMO(\ZR^d)$ and $\MM(f)$ is not identically infinite, then $\MM_\phi(f)\in \BLO(\ZR^d)$ and 
	\begin{equation}\label{x63}
	\|\MM_\phi (f)\|_{\BLO(\ZR^d)}\le c_dI(\phi)\|f\|_{\BMO(\ZR^d)},
\end{equation}
where $c_d$ is a constant depending only on the dimension $d$.
\end{theorem}
In addition to the function $\phi$ being constant on the unit ball $B=\{|x|\le 1\}$, we may require the same bound \e{x63} for the non-tangential maximal function
\begin{equation}\label{r55}
	\overline\MM_\phi f(x)=\sup_{r>0, \, |u|<r}\frac{1}{r^d}\int_{\ZR^d}|f(t)|\phi\left(\frac{x+u-t}{r}\right)dt.
\end{equation}
Observe that choosing $\phi=\ZI_{\{|x|<1\}}$, the operators $M_\phi$ and $\overline\MM_\phi $ coincide with the centered and the uncentered maximal operators on $\ZR^d$ respectively. Other corollaries of the main result will be considered in \sect{S6}.

\subsection{Definition of ball-bases}
We will work on abstract measure spaces equipped with a ball-basis. The concept of ball-basis was introduced in \cite{Kar3, Kar1}. Its definition is a selection of basic properties of classical balls (or cubes) in $\ZR^n$ that are crucial in the study of singular operators on $\ZR^n$. The same papers have also introduced a class of bounded oscillation (BO) operators, unifying many operators in harmonic analysis (such as Calder\'on-Zygmund operators, the maximal function homogeneous and non-homogeneous martingale transforms, Carleson operators, etc.). BO operators were studied in these papers in the context of weighted estimates, sparse domination, good-$\lambda$ and exponential decay inequalities. We would like to refer also \cite{Ming} for a multilinear generalization of BO operators and those properties. 

Working with general ball-bases enables us to unify different maximal operators (centered, uncentered, dyadic or martingale) in one, rather than consider each operator separately. Note that there is no equivalence principles between different maximal operators (such as centered, uncentered, dyadic) on spaces $\BMO$ or $\BLO$, since those spaces are not Banach latices. Even the boundedness of the uncentered maximal operator on $\BMO(\ZR^d) $ does not imply that of for the centered one. Moreover, the classical approach provided in \cite{BDS} is applicable only for the uncentered maximal operator. 

Now recall the definition of ball-basis.
\begin{definition} Let $(X,\ZM, \mu)$ be a measure space with a $\sigma$-algebra $\ZM$ and a measure $\mu$. A family of measurable sets $\ZB$ is said to be a ball-basis if it satisfies the following conditions: 
	\begin{enumerate}
		\item[B1)] $0<\mu(B)<\infty$ for any ball $B\in\ZB$.
		\item[B2)] For any points $x,y\in X$ there exists a ball $B\ni x,y$.
		\item[B3)] If $E\in \ZM$, then for any $\varepsilon>0$ there exists a (finite or infinite) sequence of balls $B_k$, $k=1,2,\ldots$, such that $\mu(E\bigtriangleup \cup_k B_k)<\varepsilon$.
		\item[B4)] For any $B\in\ZB$ there is a ball $ B^*\in\ZB $ (called {\rm hull-ball} of $B$), satisfying the conditions
		\begin{align}
			&\bigcup_{A\in\ZB:\, \mu(A)\le 2\mu(B),\, A\cap B\neq\varnothing}A\subset  B^*,\label{h12}\\
			&\qquad\qquad\mu\left(B^*\right)\le \ZK\mu(B),\label{h13}
		\end{align}
		where $\ZK$ is a positive constant independent of $B$. 
	\end{enumerate}
\end{definition}
Ball-bases can additionally have the following properties.
\begin{definition}\label{DD}
	We say that a ball-basis $\ZB$ is doubling if there is a constant $\eta>2$ such that for any ball $B$ with $B^*\neq X$ one can find a ball $B'\supset B$ such that
	\begin{align}
		2\mu(B)\le \mu(B')\le\eta  \cdot \mu(B).\label{h73}
	\end{align}
\end{definition}
\begin{definition}\label{D3}
	A ball basis $\ZB$ is said to be regular if there is a constant $0<\theta<1$ such that for any two balls $B$ and $A$, satisfying $\mu(B)\le \mu(A)$, $B\cap A\neq \varnothing$ we have $\mu(B^*\cap A)\ge \theta \mu(B^*)$.
\end{definition}

Let us state some basic examples of ball-bases:
\begin{enumerate}
\item The family of Euclidean balls (or cubes) in $\ZR^d$ is a doubling ball-basis.  
\item Dyadic cubes in $\ZR^d$ form a doubling ball-basis.
\item Let $\ZT= \ZR/2\pi$ be the unit circle. The family of arcs in $\ZT$ is a doubling ball-basis.
\item The families of metric balls in measure spaces of homogeneous type form a ball-basis with the doubling condition (see \cite{Kar3}, sec. 7,  for the proof)
\item Let $(X,\ZM, \mu)$ be a measure space and $\{\ZB_n:\, n\in\ZZ\}$ be collections of measurable sets (called filtration) such that
\begin{itemize}
	\item each $\ZB_n$ forms a finite or countable partition of $X$,
	\item each $A\in \ZB_n$ is a union of some sets of $\ZB_{n+1}$,
	\item the collection $\ZB=\cup_{n\in\ZZ}\ZB_n$ generates the $\sigma$-algebra $\ZM$,
	\item for any points $x,y\in X$ there is a set $A\in X$ such that $x,y\in A$.
\end{itemize}
For any $A\in \ZB_n$ we denote by $\pr(A)$ the parent-ball of $A$, that is the unique element of $\ZB_{n-1}$, containing $A$. One can easily check that $\ZB$ satisfies  the ball-basis conditions B1)-B4), where for $A\in \ZB$ the hull-ball $A^*$ is the maximal element in $\ZB$, satisfying $\mu(A^*)\le 2\mu(A)$. Observe that $A^*$ can coincide with $A$ and it occurs when $\mu(\pr(A))>2\mu(A)$.
Besides, a martingale ball-basis is doubling if and only if $\mu(\pr(A) )\le c\mu(A)$ for any $A\in \ZB$ and for some constant $c>0$.
\end{enumerate}
We also note that all these ball-bases are regular (see \df{D3})
\begin{definition}\label{D2}
A family of non-negative real valued functions
\begin{equation}\label{x47}
	\{\phi_B:\, B\in \ZB\}\subset L^1(X)\cap L^\infty(X)
\end{equation}
is said to be kernel-structure if there are positive constants $c_0$, $c_1$, $c_2$ and $c_3$such that
	\begin{enumerate}
		\item[K1)] For every ball $B\in \ZB$ we have 
		\begin{equation}\label{x45}
			\int_{X}\phi_B=1.
		\end{equation}
\item[K2)] For any ball $B\in \ZB$ we have the inequalities
\begin{equation}\label{y4}
	c_1\cdot \frac{\ZI_B(x)}{\mu(B)}\le\phi_B(x)\le c_2\cdot \frac{1}{\mu(B)}\omega\left(\frac{d(x,B)}{\mu(B)}\right), \quad x\in X  (\hbox{ see }\e{x48}).
\end{equation}
where $\omega:[1,\infty)\to [0,1]$ is a modulus of continuity, i.e. it is a non-increasing function such that $\omega(1)=1$ and $\omega(2x)\le c_0\cdot \omega(x)$.
\end{enumerate}
\end{definition}
\begin{definition}\label{D4}
	Let for any $x\in X$ we are given a families of balls $\ZB(x)\subset \ZB$.
	The collection 
	\begin{equation}\label{x64}
		\{\ZB(x):\, x\in X\}
	\end{equation}
	is said to be a basis-structure if
	\begin{enumerate}
		\item[G1)] $x\in B$ for any $B\in \ZB(x)$, 
		\item[G2)]  for any ball $B\ni x$ (not necessarily from $\ZB(x)$) there is a ball $\bar B\in \ZB(x)$ with $\bar B\supset B,\quad \mu(\bar B)\le \eta \mu(B)$,
		where $\eta \ge 1$ is a constant,
	\end{enumerate}
\end{definition}
Examples of kernel and basis structures will be considered in \sect{S6}. Given a kernel-structure and a basis-structure define a kernel-basis couple (KB-couple)
\begin{equation}
	\ZG=\big\{\{\phi_B:\, B\in \ZB\},\{\ZB(x):\, x\in X\}\big\}.
\end{equation}
Any KB-couple $\ZG$ defines a maximal operator (see \e{y6}), which will be the main object of study of the present paper.
\subsection{Notations}
Denote by $L_{loc}(X)$ the space of functions (called locally integrable), which have finite integrals on any ball $B\in \ZB$, and let $L^0(X)$ be the space of almost everywhere finite measurable functions on $X$. We set 
\begin{align}
	&f_B=\frac{1}{\mu(B)}\int_Bf,\\
	&	\langle f\rangle_B=\frac{1}{\mu(B)}\int_B|f|, \\
	&\langle f\rangle_B^*=\sup_{A\in \ZB:A\supset B}\langle f\rangle_{A},\label{y56}\\
	&\langle f\rangle_{\#,B}=\langle f-f_B\rangle_B=\frac{1}{\mu(B)}\int_B|f-f_B|,
\end{align}
\begin{align}
	& \langle f\rangle_{\#, B}^*=\sup_{A\in \ZB:A\supset B}\langle f\rangle_{\#,A},\label{x5}\\ 
	&\SUP_E(f)=\esssup_{x\in E}|f(x)|,\\
	& \INF_E(f)=\essinf_{x\in E}|f(x)|,\label{y98}\\
	&\OSC_E(f)=\esssup_{x,x'\in E}|f(x)-f(x')|=\SUP_E(f)-\INF_E(f),\label{y97}\\
	&d(x,B)=\inf_{A\in \ZB:\, A\supset B\cup \{x\}}\mu(A).\label{x48}\\
	&\OSC_{B,\alpha}(f)=\inf_{E\subset B:\, \mu(E)> \alpha\mu(B)}\OSC_E(f),\quad 0<\alpha <1,\label{y100}\\
	&\LOSC_{B,\alpha}(f)=\inf_{\stackrel{E\subset B:\, \mu(E)> \alpha\mu(B)}{\INF_E(f)=\INF_B(f)}}\OSC_E(f),\quad 0<\alpha <1,\\
	&\langle f\rangle_{\phi_B}=\int_X|f|\phi_B,
\end{align}
where  $f\in L_{loc}(X)$,  $A, B\in \ZB$ and $E\subset X$ is a measurable set.
The $\alpha$-oscillation functional $\OSC_{B,\alpha}(f)$ can be equivalently defined as follows: that is $\OSC_{B,\alpha}(f)=\inf_{a<b}(b-a)$, where the infimum is taken over all the numbers $a<b$ satisfying $\mu\{x\in B:\, f(x)\in [a,b]\}>\alpha\mu(B)$. Considering only intervals $[a,b]$ with $a=\INF_B(f)$, we similarly get the "lower" $\alpha$-oscillation $\LOSC_{B,\alpha}(f)$. Clearly we have 
\begin{equation}
	\OSC_{B,\alpha}(f)\le \LOSC_{B,\alpha}(f).
\end{equation}
In the sequel positive constants depending only our ball-basis $\ZB$ and the KB-couple $\ZG$ will be called admissible constants. The relation $a\lesssim b$ will stand for $a\le c\cdot b$, where $c>0$  is admissible. We write $a\sim b$ if the relations  $a\lesssim b$ and $b\lesssim a$ hold simultaneously. The notation $\log a$ will stand for $\log_2 a$.
\subsection{Main results}
Let $(X,\mu)$ be a measure space equipped with a ball-basis $\ZB$ and and suppose that $\ZG$ is a KB-couple. For any function $f\in L_{loc}(X)$ (locally integrable on $X$) consider the maximal operator
\begin{equation}\label{y6}
	\MM_{\ZG} f(x)=\sup_{B\in \ZB(x)}\langle f\rangle_{\phi_B}.
\end{equation}
In general, $\MM_{\ZG} (f)$ need not to be finite for all  locally integrable functions $f$, as well as it can also be non-measurable. We just note that in the above notations (in particular in $\OSC_{B,\alpha}$ and $\LOSC_{B,\alpha}$) one can allow to consider a non-measurable function $f$. Moreover, in some places below we will also need to integrate non-measurable functions and work with non-measurable function. For this purpose, in \sect{S2} we will give definitions of the outer measure and $L^p$-spaces of non-measurable functions. 

Let us now state the main result of the paper.
\begin{theorem}\label{T2}
	Let $\ZB$ be a doubling ball-basis and let $\omega:[1,\infty)\to [0,1]$ be a modulus of continuity, satisfying 
	\begin{equation}\label{y7}
		I(\omega)=1+\int_1^\infty\omega(t)\log(1+t)dt<\infty.
	\end{equation}
	If $f\in L_{loc}(X)$ and the maximal function $\MM (f)$ is not identically infinite, then for any ball $B\in \ZB$ with $\langle f\rangle^*_{\#,B}<\infty$, we have
	\begin{equation}\label{2}
	\LOSC_{B,\alpha}(\MM_{\ZG} (f))\lesssim I(\omega)(1-\alpha)^{-1}\langle f\rangle^*_{\#,B}
	\end{equation}
	for any $0<\alpha<1$. 
\end{theorem}

\begin{remark}
	An inequality like \e{2} for the standard maximal function 
	\begin{equation}\label{1-1}
		\MM f(x)=\sup_{B\in \ZB:\, x\in B}\langle f\rangle_B=\sup_{B\in \ZB:\, x\in B} \frac{1}{\mu(B)}\int_B|f|
	\end{equation}
	was proved in \cite{Kar1}, where instead of the lower oscillation $\LOSC_{B,\alpha}(\MM(f))$ the standard oscillation $\OSC_{B,\alpha}(\MM(f))$ was considered.
\end{remark}
\begin{corollary}\label{KC2}
		Let $\ZB$ be a doubling ball-basis and suppose that $\omega:[1,\infty)\to [0,1]$ is a modulus of continuity, satisfying \e{y7}. If $f\in \BMO(X)$, and $\MM(f)$ is not identically infinite, then $\MM_\ZG(f)\in \BLO(X)$ and 
	\begin{equation}
		\|\MM_{\ZG}(f)\|_{\BLO(X)}\lesssim I(\omega)\|f\|_{\BMO(X)}.
	\end{equation}
Moreover, if $\mu(X)<\infty$, then we can say $\MM_\ZG$ boundedly maps $\BMO(X)$ into $\BLO(X)$.
\end{corollary}
\begin{remark}
	It is an open question whether condition \e{y7} in the statements either of Theorems \ref{T1}, \ref{T2}  or  \cor{KC2} can be replaced by $\int_1^\infty\omega(t)dt<\infty$.
\end{remark} 
\cor{KC2} is obtained by a combination of \trm{T2} and the following John-Nirenberg type theorems. For those statements,  first recall the definitions of spaces $\BMO_\alpha$, $\BLO_\alpha$, $\BMO$ and $\BLO$, which are the families of functions $f\in L^0(X)$ with the norms
\begin{align}
	&\|f\|_{\BMO_\alpha}=\sup_{B\in \ZB}\OSC_{B,\alpha}(f)<\infty,\\
	&\|f\|_{\BLO_\alpha}=\sup_{B\in \ZB}\LOSC_{B,\alpha}(f)<\infty,\\
	&\|f\|_{\BMO}=\sup_{B\in \ZB}\langle f\rangle_{\#,B}=\sup_{B\in \ZB}\langle f-f_B\rangle_{B}<\infty,\\
	&\|f\|_{\BLO}=\sup_{B\in \ZB}\langle f-\INF_B(f)\rangle_{B}<\infty,
\end{align}
respectively. The following relations between these norms can be easily checked:
\begin{align*}
	&\|f\|_{\BMO_\alpha}\le\|f\|_{\BLO_\alpha},\quad  \|f\|_{\BMO}\le 2\|f\|_{\BLO}\\
	&\|f\|_{\BMO_\alpha}\le 2(1-\alpha)^{-1}\|f\|_\BMO,\\
	&\|f\|_{\BLO_\alpha}\le 2(1-\alpha)^{-1}\|f\|_\BLO,
\end{align*}
whereas the followings are more delicate properties of these norms.
\begin{theorem}\label{T4}
For an admissible constant $\alpha\in (1/2,1)$ we have  $ \BLO_\alpha(X)= \BLO(X)$. Moreover, for any $f\in \BLO_\alpha(X)$,
	\begin{equation}\label{x6}
		\|f\|_{\BLO}\sim	\|f\|_{\BLO_\alpha}.
	\end{equation}
\end{theorem}
\begin{theorem}\label{T11}
There is an admissible constant $\alpha\in (1/2,1)$ such that  $ \BMO_\alpha(X)= \BMO(X)$. Moreover, for any $f\in \BMO_\alpha(X)$ we have $\|f\|_{\BMO}\sim	\|f\|_{\BMO_\alpha}$.
\end{theorem}

\begin{remark}
	\trm{T11} is an extension of the analogous result on $\ZR^d$ proved in \cite{John, Str}. As for \trm{T4}, to the best of our knowledge it is new even in the classical situations. The proofs of Theorems \ref{T4} and \ref{T11} are based on \pro{P} (see \sect{S5}), which is itself interesting. In particular, from \pro{P} it follows that:
if $f\in \BMO(\ZR^d)$, then for any ball $B\subset \ZR^d$,
\begin{align}
	|\{x\in B:\, |f(x)-f_B|>&c(n+1)\|f\|_\BMO\}|\\
	&\le \frac{1}{2}\cdot |\{x\in B:\, |f(x)-f_B|>cn\|f\|_\BMO\}|\label{x57}
\end{align}
where $c>0$ is an absolute constant. Clearly, this immediately implies the classical John-Nirenberg inequality \e{JN},  but the converse statement is not true at all. Inequalities like \e{x57} can also be stated for functions in $\BMO_\alpha$, $\BLO_\alpha$ and $\BLO$. In the case of $\BLO$ one just need to replace $f_B$ by $\INF_B(f)$ in \e{x57}.  The $\BLO$ version of \e{x57} implies also a John-Nirenberg type inequality for $\BLO$ spaces. The latter for the ball-basis of cubes in $\ZR^d$ was recently proved by D.~Wang et al. in \cite{Wang}.
\end{remark}
\begin{remark}
	G.~Dafni et al. in \cite{DaGi} considered Hardy-Littlewood type maximal function \e{HL} over families of engulfing bases $\ZF$ (introduced in \cite{DG}) defined as follows: 
	\begin{enumerate}
		\item for each $B\in \ZF$ there is a $\tilde B\in \ZF$ such that $|\tilde B|\le c |B|$, 
	\item for each $B\in \ZF$, with $\tilde B$ chosen in (i), if $G\in \ZF$ satisfies $G\cap B=\varnothing$, $G\cap \tilde B^c=\varnothing$, then there exists a $\bar G\in \ZF$ such that $\bar G\supset \tilde B\cup G$ and $|\bar G|\le c|G|$. 
	\end{enumerate}
It was proved in \cite{DaGi} that such a maximal function maps boundedly $\BMO$ into $\BLO$, where both spaces are defined with respect to the basis $\ZF$. Clearly families of classical balls (cubes) as well as dyadic cubes in $\ZR^d$ form engulfing bases. Moreover, one can check that doubling ball-bases defined above are engulfing basis too. In this paper the consideration of ball-bases enables obtaining $\BMO-\BLO$ boundedness for convolution type maximal operators. 
\end{remark}
\begin{remark}
	Note that for the proof of \trm{T2} condition B3) as well us the regularity condition (see \df {D3}) of the ball-basis are not used, but those are significant in the proofs of Theorems \ref{T4}, \ref{T11} and  \cor{KC2}. 
\end{remark}

\section{Outer measure and $L^p$-norms of non-measurable functions}\label{S1}
It is well known for classical examples of ball-bases that the union of any collections of balls is measurable and such a property is important to ensure measurability of certain maximal operators. For general ball-bases non-countable union of balls need not to be measurable. So in some cases we will need to use outer measure generated by the basic measure $\mu$. Hence, let $(X,\ZM,\mu)$ be a measure space. Define the outer measure of a set $E\subset X$ by
\begin{equation*}
	\mu^*(E)=\inf_{F\in \ZM:\, F\supset E}\mu(F).
\end{equation*}
For any function $f\in L^0(X)$ we denote
\begin{align*}
	&G_f(t)=\{x\in X:\, |f(x)|>t\},\quad t\ge 0,\\
	& \lambda_f(t)=\mu^*\left(G_f(t)\right).
\end{align*}
Observe that $\lambda_f(t)$ is a distribution function, i.e. it is right-continuous. 
For arbitrary $f\in L^0(X)$ we define
\begin{align}
	&\|f\|_{L^p}=\left(p\int_0^\infty t^{p-1}\lambda_f(t)dt\right)^{1/p},\label{r52}\\
	&\|f\|_{L^{p,\infty}}=\sup_{t>0}t(\lambda_f(t))^{1/p}.\label{r54}
\end{align}
One can easily check that the standard triangle and  H\"{o}lder inequalities as well as the Marcinkiewicz interpolation theorem hold in such setting of $L^p$ norms. We will use below weak-$L^1$ bound of the standard maximal function \e{1-1}.
\begin{theorem}[\cite{Kar3}, Theorem 4.1]\label{T1-1}
	The maximal function \e {1-1} satisfies weak-$L^1$ inequality. Namely, 
	\begin{equation}\label{h38}
		\mu^*\{x\in X:\,\MM f(x)>\lambda\}|\lesssim \|f\|_{L^1(X)}/\lambda,\quad \lambda>0.
	\end{equation} 
\end{theorem}
\section{Preliminary properties of ball bases}\label{S2}
Some of lemmas proved in this section are versions of similar statements from \cite{Kar3, Kar1}. We find necessary to state those proof completely.  Let $\ZB$ be a ball-basis in the measure space $(X,\mu)$. From B4) condition it follows that if balls $A,B$ satisfy $A\cap B\neq\varnothing$, $\mu(A)\le 2\mu(B)$, then $A\subset B^*$. This property will be called two balls relation. We say a set $E\subset X$ is bounded if $E\subset B$ for a ball $B\in\ZB$.
\begin{lemma}\label{L0}
	If $(X,\mu)$ is a measure space equipped with a ball-basis $\ZB$ and $\mu(X)<\infty$, then $X\in \ZB$.
\end{lemma}
\begin{proof}
	Fix an arbitrary point $x_0\in X$ and let $\ZA\subset \ZB$ be the family of balls containing $x_0$. By B2) condition we can write $X=\cup_{A\in \ZA}A$. Then there exists a $B\in \ZA$ such that $\mu(B)>\frac{1}{2} \sup_{A\in \ZA}\mu(A)$ and using B4) condition we obtain
	\begin{equation*}
		X=\cup_{A\in \ZA}A\subset B^*.
	\end{equation*}
	Thus $X=B^*\in \ZB$.
\end{proof}
\begin{lemma}\label{L5}
	Let $(X,\mu)$ be a measure space equipped with a ball-basis $\ZB$ and $G\in \ZB$. Then there exists a sequence of balls $G=G_1,G_2, \ldots$ (finite or infinite) such that 
	\begin{equation}
		X=\cup_kG_k,\quad G_k^*\subset G_{k+1}. 
	\end{equation}
Moreover for any $B\in \ZB$ there is a ball $G_n\supset B$.  If the ball-basis is doubling, then we can additionally claim 
\begin{equation}\label{x71}
	2\mu(G_k)\le \mu(G_{k+1})\le \beta\mu(G_k), 
\end{equation}
with an admissible constant $\beta>2$. 
\end{lemma}

\begin{proof}
	Choose an arbitrary point $x_0\in G$ and let $\ZA\subset \ZB$ be the family of balls containing the point $x_0$. Choose a sequence $\eta_n\nearrow\eta=\sup_{A\in \ZA}\mu(A)$, such that $\eta_1<\mu(G)$, where $\eta$ can also be infinity. Let us see by induction that there is an increasing  sequence of balls $G_n\in \ZA$ such that $G_1=G$, $\mu(G_n)> \eta_n$ and $G_n^*\subset G_{n+1}$. The base of induction is trivial. Suppose we have already chosen the first balls $G_k$, $k=1,2,\ldots, l$. There is a ball $A\in \ZA$ so that $\mu(A)>\eta_{l+1}$. Let $C$ be the biggest in measure among two balls $A$, $G_l^*$ and denote $G_{l+1}=C^*$. By property B4) we get $A\cup G_l^*\subset C^*=G_{l+1}$. This implies $\mu(G_{l+1})\ge \mu(A)>\eta_{l+1}$ and $G_{l+1}\supset G_l^*$, completing the induction.
	Let us see that $G_n$ is our desired sequence of balls. Indeed, let $B$ be an arbitrary ball. By B2) property there is a ball $A$ containing $x_0$, such that $A\cap B\neq \varnothing$. Then by property B4) we may find a ball $C\supset A\cup B$. For some $n$ we will have $\mu(C)\le 2\mu(G_n)$ and so once again using B4), we get $B\subset  C\subset G_{n}^*\subset G_{n+1}$. 
	
	The second part of the lemma can be proved by a similar argument. Using the doubling condition (\df {DD}) one can easily find a sequence of balls $G=G_1,G_2, \ldots$ such that
	\begin{equation}
		G_k^*\subset G_{k+1},\quad 2\mu(G_k^*)\le \mu(G_{k+1})\le \eta \mu(G_k^*)\le \eta\ZK \mu(G_k).
	\end{equation}
If $\mu(X)<\infty$, then clearly the ball sequence will be finite and its last term say $G_n$ will coincide with $X$. So the proof follows. If $\mu(X)=\infty$, then our ball sequence is infinite and $\mu(G_n)\to \infty$. As in the first part of the proof, for an arbitrary ball $B$ we can find a ball $C\supset B$ such that $C\cap G\neq \varnothing$. There is a ball $G_n$ such that $\mu(G_n)> \mu(C)$ and we get $B\subset  C\subset G_{n}^*\subset G_{n+1}$. This completes the proof of lemma.
\end{proof}
\begin{remark}
	 If $\mu(X)<\infty$, then the sequence $G_n$ in \lem{L5} is finite and its last-most term coincides with $X$.
\end{remark}

\begin{lemma}\label{L1-1}
	Let $(X,\mu)$ be a measure space with a ball bases $\ZB$. If $E\subset X$ is bounded and a family of balls $\ZG $ is a covering of $E$, i.e. $E\subset \bigcup_{G\in \ZG}G$, then there exists a finite or infinite sequence of pairwise disjoint balls $G_k\in \ZG$ such that $E \subset \bigcup_k G_k^{*}$.
\end{lemma}
\begin{proof}
	The boundedness of $E$ implies $E\subset B$ for some $B\in \ZB$. If there is a ball $G\in\ZG$ so that $G\cap B\neq \varnothing$, $\mu(G)> \mu(B)$, then by two balls relation we will have $E\subset B\subset G^{*}$. Thus our desired sequence can be formed by a single element $G$. Hence we can suppose that every $G\in\ZG$ satisfies $G\cap B\neq \varnothing$, $\mu(G)\le \mu(B)$ and again by two balls relation $G\subset B^*$. Therefore, $\bigcup_{G\in\ZG}G\subset B^{*}$.
	Choose $G_1\in \ZG$, satisfying $\mu(G_1)> \frac{1}{2}\sup_{G\in\ZG}\mu(G)$. Then, suppose by induction we have already chosen elements $G_1,\ldots,G_k$ from $\ZG$. Choose $G_{k+1}\in \ZG$  disjoint with the balls $G_1,\ldots,G_k$ such that
	\begin{equation}\label{b1}
		\mu(G_{k+1})> \frac{1}{2}\sup_{G\in \ZG:\,G\cap G_j=\varnothing,\,j=1,\ldots,k}\mu(G).
	\end{equation}
	If for some $n$ we will not be able to determine $G_{n+1}$ the process will stop and we will get a finite sequence $G_1,G_2,\ldots, G_n$. Otherwise our sequence will be infinite. We shall consider the infinite case of the sequence (the finite case can be done similarly). Since the balls $G_n$ are pairwise disjoint and $G_n\subset B^{*}$, we have $\mu(G_n)\to 0$. Choose an arbitrary $G\in\ZG$ with $G\neq G_k$, $k=1,2,\ldots $ and let $m$ be the smallest integer satisfying
	$\mu(G)\ge 2\mu(G_{m+1})$.
	So we have $G \cap G_j\neq\varnothing$
	for some $1\le j\le m$, since otherwise by \e{b1}, $G$ had to be chosen instead of $G_{m+1}$. Besides, we have $\mu(G)< 2\mu(G_{j})$ because of the minimality property of $m$, and so by two balls relation $G\subset G_{j}^{*}$. Since $G\in \ZG$ was chosen arbitrarily, we get $E \subset\bigcup_{G\in \ZG} G\subset \bigcup_k G_k^{*}$.
\end{proof}

\begin{definition}\label{D1}
	For a measurable set $E\subset X$ a point $x\in E$ is said to be density point if for any $\varepsilon>0$ there exists a ball $B\ni x$ such that
	\begin{equation*}
		\mu(B\cap E)>(1-\varepsilon )\mu(B).
	\end{equation*} 
	We say a ball basis satisfies the density property if for any measurable set $E$ almost all points $x\in E$ are density points. 
\end{definition}
\begin{lemma}\label{L12}
	Every ball basis $\ZB$ satisfies the density condition.
\end{lemma}
\begin{proof}
	Applying \lem{L5} one can check that it is enough to establish the density property for the bounded measurable sets. Suppose to the contrary there exist a bounded measurable set $E\subset X$ together its subset $F\subset E$ (maybe non-measurable) with an outer measure $\mu^*(F)>0$
	such that
	\begin{equation}\label{z38}
		\mu(B\setminus E)>\alpha\mu(B) \text{ whenever }B\in\ZB,\, B\cap F\neq\varnothing,
	\end{equation}
	where $0<\alpha<1$. According to the definition of the outer measure one can find a measurable set $\bar F$ such that 
	\begin{equation}\label{z2}
		F\subset \bar F\subset E, \quad \mu(\bar F)=\mu^*(F).
	\end{equation}
	By B3)-condition there is a sequence of balls $B_k$, $k=1,2,\ldots $, such that
	\begin{equation}\label{z42}	
		\mu\left(\bar F\bigtriangleup (\cup_k B_k)\right)<\varepsilon.
	\end{equation}
	We discuss two kind of balls $B_k$, satisfying either $B_k\cap F=\varnothing$ or $B_k\cap F\neq\varnothing$. For the first collection we have 
	\begin{equation*}
		\mu^*(F)\le \mu\left(\bar F\setminus \bigcup_{k:\,B_k\cap F=\varnothing} B_k\right)\le \mu(\bar F)=\mu^*(F).
	\end{equation*}
	This implies 
	\begin{equation*}
		\mu\left(\bar F\bigcap \left(\bigcup_{k:\,B_k\cap F=\varnothing} B_k\right)\right)=0
	\end{equation*}
	and therefore, combining also \e{z42}, we obtain
	\begin{equation}\label{z3}
		\mu\left(\bigcup_{k:\,B_k\cap F=\varnothing} B_k\right)<\varepsilon.
	\end{equation}
	Now consider the balls, satisfying $B_k\cap F\neq\varnothing$. It follows from \e {z38} and \e{z2} that 
	\begin{equation}\label{z41}
		\mu(B_k\setminus \bar F)\ge \mu(B_k\setminus E)>\alpha \mu(B_k),\quad k=1,2,\ldots.
	\end{equation}
	Since $E$ and so $\bar F$ are bounded, applying \lem {L1-1} and \e{z42}, one can find a subsequence of pairwise disjoint balls $\tilde B_k$, $k=1,2,\ldots$, such that
	\begin{equation*}
		\mu\left(\bar F\setminus\cup_k\tilde B_k^{*}\right)<\varepsilon.
	\end{equation*} 
	Thus, from B4)-condition, \e{z42}, \e {z3} and \e {z41}, we obtain
	\begin{align*}
		\mu^*(F)< \mu(\bar F)&\le \mu\left(\cup_k\tilde B_k^{*}\right)+\varepsilon\le \ZK\sum_k\mu(\tilde B_k)+\varepsilon\\
		&= \ZK\mu\left(\bigcup_{k:\,\tilde B_k\cap F=\varnothing} \tilde B_k\right)+\ZK\sum_{k:\,\tilde B_k\cap F\neq\varnothing} \mu(\tilde B_k)+\varepsilon\\
		&<\ZK\varepsilon+\frac{\ZK}{\alpha}\sum_k\mu(\tilde B_k\setminus \bar F)+\varepsilon\\
		&\le (\ZK+1)\varepsilon+\frac{\ZK}{\alpha}\mu\left(\bar F\bigtriangleup( \cup_k B_k)\right)<\varepsilon\left(2\ZK+1+\frac{\ZK}{\alpha}\right).
	\end{align*}
	Since $\varepsilon $ can be arbitrarily small, we get $\mu^*(F)=0$ and so a contradiction.
\end{proof}

\section{Proof of \trm{T2}}
\begin{lemma}\label{LK1}
	If $A\subset B$ are balls and $x\notin B^*$, then $d(x,A)\ge\mu(B)$.
\end{lemma}
\begin{proof}
	Suppose to the contrary that $d(x,A)< \mu(B)$, which means there is a ball $A'\supset A$ such that $x\in A'$ and $\mu(A')<\mu(B)$. Then we get $A'\subset B^*$, which is in contradiction with $x\notin B^*$.
\end{proof}
The following two lemmas are standard and well-known in the classical situations.
\begin{lemma}\label{L6}
	For any function $f\in L^1_{loc}(X)$ and balls $A, B$ with $A\cap B\neq\varnothing$, $\mu(A)\le \mu(B)$ we have
	\begin{equation}\label{x44}
		|f_A-f_B|\lesssim \frac{\mu(B)}{\mu(A)}\cdot  \langle f\rangle_{\#,A}^*.
	\end{equation}
	\begin{proof}
		First suppose that $A\subset B$. Then we get
		\begin{align}
			\left|f_{A}-f_{B}\right|&\le \frac{1}{\mu(A)}\int_{A}|f-f_{B}|\\
			&\le\frac{\mu(B)}{\mu(A)} \cdot \frac{1}{\mu(B)}\int_{B}|f-f_{B}|\le \frac{\mu(B)}{\mu(A)}\cdot \langle f\rangle_{\#,B}.\label{x72}
		\end{align}
	In the general case, by the two ball relation we can write $A\subset B^*$ and so applying \e{x72} we obtain
	\begin{align}
	\left|f_{A}-f_{B}\right|&\le \left|f_{A}-f_{B^*}\right|+\left|f_{B}-f_{B^*}\right|\\
	&\le \left(\frac{\mu(B^*)}{\mu(A)}+\frac{\mu(B^*)}{\mu(B)}\right) \langle f\rangle_{\#,B^*}\lesssim \frac{\mu(B)}{\mu(A)}\cdot  \langle f\rangle_{\#,A}^*.
	\end{align}
	\end{proof}
\end{lemma}
\begin{lemma}\label{L10}
	Let $\ZB$ be a doubling ball-basis on $X$. Then for any $f\in L^1_{loc}(X)$ and balls $A\subset B$ it holds the inequality
	\begin{equation}\label{x42}
		\langle f-f_A\rangle_B\lesssim (1+\log(\mu(B)/\mu(A)))\cdot \langle f\rangle_{\#,A}^*.
	\end{equation}
\end{lemma}
\begin{proof} 
Applying the second part of \lem{L5}, we find a sequence of balls $A=G_0\subset G_1\subset \dots$, satisfying \e{x71}. We have
$\mu(G_n)<\mu(B)\le \mu(G_{n+1})$ for some integer $n$ and we can write $n\lesssim 1+\log(\mu(B)/\mu(A))$.
	By \lem {L6}, 
	\begin{align}
		&|f_{G_k}-f_{G_{k+1}}|\lesssim \langle f\rangle_{G_{k},\#}^*\lesssim \langle f\rangle_{\#,A}^* \hbox{ (see } \e{x5}),\\
		&|f_{G_n}-f_{B}|\lesssim \langle f\rangle_{G_n,\#}^*\lesssim\langle f\rangle_{\#,A}^*,
	\end{align}
	so we obtain
	\begin{align*}
		\langle f-f_A\rangle_B&\le \langle f-f_B\rangle_B+|f_A-f_B|\\
		&\le \langle f-f_B\rangle_B+ \sum_{k=0}^{n-1}|f_{G_k}-f_{G_{k+1}}|+|f_{G_n}-f_{B}|\\
		&\lesssim n\langle f\rangle_{\#,A}^*\lesssim (1+\log(\mu(B)/\mu(A)))\langle f\rangle_{\#,A}^*.
	\end{align*}
\end{proof}

We say a positive function $\phi_B$ is comparable with a ball $B$ if it satisfies \e{y4}. The simplest function comparable with a ball $B$ is the $\ZI_B/\mu(B)$.
\begin{lemma}\label{LK2}
	Let $A,B\in \ZB$, $A\subset B$, be balls of a doubling ball-basis $\ZB$ and let $\phi_A$, $\phi_B$ be functions comparable with $A$ and $B$ respectively. Then for any function $f\in L_{loc}(X)$ with  $\langle f\rangle_{\#,A}^*<\infty$ both $f_{\phi_A}$ and $f_{\phi_B}$ are finite and 
	\begin{equation}\label{y2}
		\langle f-f_{\phi_A}\rangle_{\phi_B}\lesssim I(\omega)(1+\log(\mu(B)/\mu(A)) )\cdot \langle f\rangle_{\#,A}^*.
	\end{equation}
\end{lemma}
\begin{proof}
	Applying \lem{L5}, we find a sequence of balls $B=B_0,B_1,B_2,\ldots$ such that
	\begin{align}
		&2\mu(B_{n-1})\le \mu(B_{n})\lesssim \mu(B_{n-1}),\label{x3}\\
		&\cup_kB_k=X, \quad B_n^*\subset B_{n+1}
	\end{align}
	Applying  \lem{LK1}, for $x\in B_{n}\setminus B_{n-1}\subset	B_n\setminus B_{n-2}^*$ we can write
	\begin{equation}
		\mu(B_n)\ge d(x,B)\ge  \mu(B_{n-2}),\quad n\ge 2.
	\end{equation} 
	Thus, applying the right-hand side of inequality \e{y4} and the monotonicity of the function $\omega$, we obtain
	\begin{align*}
		&\phi_B(x)\lesssim \frac{1}{\mu(B)}\lesssim \frac{1}{\mu(B_1)},\quad x\in B_1,\\
		&\phi_B(x)\lesssim \frac{1}{\mu(B)}\omega\left(\frac{\mu(B_{n-2})}{\mu(B)}\right),\quad x\in B_{n}\setminus B_{n-1}, \quad n\ge 2.
	\end{align*}
	Besides,  from \e{x3} it follows that $\mu(B_{n-1}\setminus B_{n-2})\ge \mu(B_{n-2})\gtrsim \mu(B_n)$. Therefore, for every $x\in X$ we obtain
	\begin{align}\label{x40}
		\phi_B(x)&\lesssim \frac{\ZI_{B_1}(x)}{\mu(B_1)}+\frac{1}{\mu(B)}\sum_{n\ge 2}\omega\left(\frac{\mu(B_{n-2})}{\mu(B)}\right)\ZI_{B_{n}\setminus B_{n-1}}(x)\\
		&\lesssim \frac{\ZI_{B_1}(x)}{\mu(B_1)}+\sum_{n\ge 2}\frac{\mu(B_{n-1}\setminus B_{n-2})}{\mu(B)}\omega\left(\frac{\mu(B_{n-2})}{\mu(B)}\right)\frac{\ZI_{B_n}(x)}{\mu(B_n)}.\label{r56}
	\end{align}
	Applying \e{x3} and the property $\omega(2x)\lesssim \omega(x)$, one can check
	\begin{align}
		&\sum_{n\ge 2}\frac{\mu(B_{n-1}\setminus B_{n-2})}{\mu(B)}\omega\left(\frac{\mu(B_{n-2})}{\mu(B)}\right)\lesssim \int_1^\infty \omega(x)dx\lesssim I(\omega),\\
		&\sum_{n\ge 2}n\cdot \frac{\mu(B_{n-1}\setminus B_{n-2})}{\mu(B)}\omega\left(\frac{\mu(B_{n-2})}{\mu(B)}\right)\lesssim \int_1^\infty \omega(x)\log(1+x)dx<I(\omega).
	\end{align}
	Using \lem{L10} and the right side of inequality in \e{x3}, we can write
	\begin{align}
		\langle f-f_A\rangle_{B_n}&\lesssim(1+\log (\mu(B_n)/\mu(A)))\langle f\rangle_{\#,A}^*\\
		&= (1+\log(\mu(B_n)/\mu(B))+\log(\mu(B)/\mu(A))\langle f\rangle_{\#,A}^*\\
		&\lesssim (n+\log(\mu(B)/\mu(A)) )\langle f\rangle_{\#,A}^*,\quad n\ge 1. 
	\end{align}
	Hence we obtain
	\begin{align}
		\langle f-f_A\rangle_{\phi_B}&\lesssim \langle f-f_A\rangle_{B_1}+\sum_{n\ge 2}\frac{\mu(B_{n-1}\setminus B_{n-2})}{\mu(B)}\omega\left(\frac{\mu(B_{n-2})}{\mu(B)}\right) \langle f-f_A\rangle_{B_n}\\
		&\lesssim (I(\omega)+\log(\mu(B)/\mu(A)))\cdot  \langle f\rangle_{\#,A}^*,\label{k5}
	\end{align}
	From this bound first we conclude $f_{\phi_B}\le |f_A|+\langle f-f_A\rangle_{\phi_B}<\infty$. Inequality \e{k5} holds for $B=A$ too, so we can similarly write $f_{\phi_A}<\infty$. Applying \e{k5} with $B=A$ we get
	\begin{equation}\label{k6}
		|f_{\phi_A}-f_A|=\left|\int_{X}(f-f_A)\phi_A\right|\lesssim  I(\omega)\cdot \langle f\rangle_{\#,A}^*,
	\end{equation}
	then, combining \e{k5} and \e{k6}, we obtain
	\begin{align*}
		\int_{X}|f-f_{\phi_A}|\phi_B&\le \int_{X}|f-f_A|\phi_B+|f_{\phi_A}-f_A|\\
		&\lesssim( I(\omega)+\log(\mu(B)/\mu(A)))\cdot \langle f\rangle_{\#,A}^*.
	\end{align*}
Since $I(\omega)\ge 1$ (see \e{y7}), this implies \e{y2}.
\end{proof}
\begin{lemma}\label{L11}
	If $f\in L_{loc}(X)$ and $\MM f(x_0)<\infty$ at some point $x_0\in X$, then $\MM f(x)<\infty$ almost everywhere. 
\end{lemma}
\begin{proof}
	Let $B$ be an arbitrary ball containing the point $x_0$. For any point $x\in B$ and a ball $A\ni x$ with $\mu(A)> \mu(B)$ we have $x_0\in A^*$. Hence,
	\begin{equation*}
		\langle f\rangle_A\le \ZK \langle f\rangle_{A^*}\le \ZK\cdot \MM f(x_0).
	\end{equation*}
Thus we conclude 
\begin{align}
	\MM f(x)&\le \ZK\cdot \MM f(x_0)+\sup_{A\ni x,\, \mu(A)\le \mu(B)}\langle f\rangle_A\\
	&\le \ZK\cdot \MM f(x_0)+\sup_{A\ni x}\langle f\cdot \ZI_{B^*}\rangle_A\\
	&= \ZK\cdot \MM f(x_0)+\MM (f\cdot \ZI_{B^*})(x)
\end{align}
for any $x\in B$. On the other hand according to weak-$L^1$ bound we can write
\begin{equation}
	\mu\{x\in B:\, \MM (f\cdot \ZI_{B^*})(x)>\lambda\}\lesssim \frac{\int_{B^*}|f|}{\lambda}.
\end{equation}
This implies $\MM (f\cdot \ZI_{B^*})(x)<\infty$ a.e. and so we will have $\MM f(x)<\infty$ a.e. on $B$. Since $B$ is an arbitrary ball, in view of \lem{L5}, we get 
$\MM f(x)<\infty$ a.e. on $X$.
\end{proof}
\begin{lemma}
	If a function $\phi_B$ satisfies \e{y4}, then
	\begin{equation}\label{x43}
		\left|\int_Xf\cdot \phi_B\right|\lesssim \MM f(x)\text{ for any }x\in B.
	\end{equation}
\end{lemma}
\begin{proof}
	We will use inequality \e{r56} obtained in the proof of \lem{LK1}. Thus for $x\in B$ we obtain
	\begin{align*}
		\left|\int_Xf\cdot \phi_B\right|&\le \MM f(x)\left(1+\sum_{n\ge 2}\frac{\mu(B_{n-1}\setminus B_{n-2})}{\mu(B)}\omega\left(\frac{\mu(B_{n-2})}{\mu(B)}\right)\right)\\
			&\lesssim  \MM f(x)\int_1^\infty \omega(x)dx.
	\end{align*}
\end{proof}
\begin{proof}[Proof of \trm{T2}]
	We will use the standard maximal function $\MM$ defined in \e{1-1}.
Fix a nontrivial function $f\in L_{loc}(X)$ and a ball $B$ satisfying the hypothesis of theorem. Set $g=(f-f_{\phi_{B}})\cdot \ZI_{B^*}$ and for a $\lambda>0$ consider the set $E=E_{B,\lambda}=\{y\in B:\, \MM g(y)\le \lambda\}$, which can be non-measurable. According to the weak-$L^1$ bound of $\MM$ (see \trm{T1-1}) we have
	\begin{equation}
		\mu^*(B\setminus E)=\mu^*\{y\in B:\, \MM g(y)>\lambda\}\lesssim  \frac{1}{\lambda}\cdot \int_{ B^*}|g|.
	\end{equation}
	So for an appropriate number $\lambda\sim (1-\alpha)^{-1}\langle g\rangle_{B^*}$	we have $\mu^*(B\setminus E)< (1-\alpha)\mu(B)$. Therefore there is a measurable set $F\subset E$ such that 
	\begin{equation}\label{b2}
		\mu(F)>\alpha \mu(B), \quad F\subset E.
	\end{equation} 
	Since the function $\ZI_{B^*}/\mu(B^*)$ is a comparable with the ball $B^*$, by \lem{LK2} we can write $\langle f-f_{\phi_{ B}}\rangle_{ B^*}\lesssim  I(\omega)\langle f\rangle^*_{\#,B}$. Thus, for any $x\in F$ we obtain
	\begin{align}
		&\MM g(x)\lesssim (1-\alpha)^{-1}\langle g\rangle_{ B^*}=(1-\alpha)^{-1}\langle f-f_{\phi_{ B}}\rangle_{ B^*}\label{b1}\\
		&\qquad\quad\, \lesssim (1-\alpha)^{-1}I(\omega)\langle f\rangle^*_{\#,B}.
	\end{align}
Now denote
\begin{equation}
	G=\{x\in X:\, \MM f(x)<\infty\}.
\end{equation}
According to \lem{L11} and that $\MM(f)$ is not identically infinite, we have $\mu(X\setminus G)=0$.	Fix two points 
\begin{equation}\label{x70}
		x\in F\cap G,\quad x'\in B\cap G.
\end{equation}
According to \e{x43}, for any $\delta >0$ there is a ball $A\in \ZB$ such that
	\begin{equation}\label{k1}
		\MM_{\ZG} f(x)\le  \langle f\rangle_{A}+\delta\text{ and } x\in A.
	\end{equation}
	If $\mu(A)>\mu(B)$, then we will have $x'\in B\subset A^*\subset \bar B$, where $\bar B\in \ZB(x')$ and $\mu(\bar B)\lesssim \mu(A)$ (see G2)).
	Thus by the definition of $\MM_{\ZG}$ we can write
	\begin{equation}\label{k3}
		\MM_{\ZG} f(x')\ge \langle f\rangle_{\phi_{\bar B}}.
	\end{equation}
	Using the left hand side of \e{y4}, we can write $\langle f-f_{\phi_{\bar B}}\rangle_{A}\lesssim \langle f-f_{\phi_{\bar B}}\rangle_{\bar B}$. Then using also relations \e{x45}, \e{y2}, \e{k1} and \e{k3}, we obtain
	\begin{align}
		\MM_{\ZG} f(x)-\MM_{\ZG} f(x')&\le \langle f\rangle_{A}-\langle f\rangle_{\phi_{\bar B}}+\delta\\
		&\le \langle f-f_{\phi_{\bar B}}\rangle_{A}+|f_{\phi_{\bar B}}|+\langle f-f_{\phi_{\bar B}}\rangle_{\phi_{\bar B}}-|f_{\phi_{\bar B}}|+\delta\\
		&\lesssim  \langle f-f_{\phi_{\bar B}}\rangle_{\phi_{\bar B}}+\langle f-f_{\phi_{\bar B}}\rangle_{\phi_{\bar B}}+\delta\\
		&\lesssim I(\omega)\langle f\rangle^*_{\#,B}+\delta.\label{a10}
	\end{align}
 In the case $\mu(A)\le \mu(B)$ we will have $A\subset  B^*$. Then by the definition of function $g$ we can write
	\begin{equation}\label{x4}
	\langle f-f_{\phi_{B}}\rangle_{A}=\langle g\rangle_{A}.
	\end{equation}
Besides, for any $x'\in B\cap G$ we may find a ball $\bar B\supset B^*$ such that 
\begin{equation}
	\bar B\in \ZB(x'),\quad \mu(\bar B)\lesssim \mu(B).
\end{equation}
Thus we get $\MM_{\ZG} f(x')\ge \langle f\rangle_{\phi_{\bar B}}$.
Then, using \e{y2}, \e{b1}, \e{k1} and \e{x4}, we obtain
	\begin{align}\label{a11}
		\MM_{\ZG} f(x)-\MM_{\ZG} f(x')&\le \langle f\rangle_{A}-\langle f\rangle_{\phi_{\bar B}}+\delta\\
		&\le \langle f-f_{\phi_{ B}}\rangle_{A}+|f_{\phi_{ B}}|+\langle f-f_{\phi_{ B}}\rangle_{{\phi_{\bar B}}}-|f_{\phi_{B}}|+\delta\\
		&\le \langle f-f_{\phi_{B}}\rangle_{A}+\langle f-f_{\phi_{B}}\rangle_{{\phi_{\bar B}}}+\delta\\
		&\lesssim \langle g\rangle_{A}+ I(\omega)\langle f\rangle^*_{\#,B}+\delta\\
		&\le \MM g(x)+ I(\omega)\langle f\rangle^*_{\#,B}+\delta\\
		&\lesssim  (1-\alpha)^{-1} I(\omega)\langle f\rangle^*_{\#,B}+\delta.
	\end{align}
	Since $\delta$ can be arbitrary small, from \e{a10} and \e{a11} we conclude
	\begin{equation*}
		\MM_{\ZG} f(x)-\MM_{\ZG} f(x')\lesssim (1-\alpha)^{-1} I(\omega)\langle f\rangle^*_{\#,B}
	\end{equation*}
for any choice of $x$ and $x'$ with conditions \e{x70}. Thus we conclude 
\begin{equation*}
	0\le \MM_{\ZG} f(x)-\INF_B(\MM_{\ZG} (f))\lesssim (1-\alpha)^{-1} I(\omega)\langle f\rangle^*_{\#,B}\text{ for all }x\in F\cap G.
\end{equation*}
Taking into account \e{b2} as well as $\mu(X\setminus G)=0$, then we obtain
	\begin{equation}\label{b3}
		\LOSC_{B,\alpha}(\MM_{\ZG} f)\lesssim (1-\alpha)^{-1} I(\omega)\langle f\rangle^*_{\#,B},
	\end{equation}
completing the proof of theorem.
\end{proof}
\section{Proofs of Theorems \ref{T4} and \ref{T11}}\label{S5}
Let $\beta$ and $\theta$ be the constants from \lem{L5} (see \e{x71}) and from  \df{D3} respectively.	
\begin{proposition}\label{P}
	Let $\ZB$ satisfy doubling and regularity properties,  $0<\varepsilon<1$ and
	\begin{equation}\label{x74}
		\alpha= 1-\frac{\varepsilon \theta}{4\beta^2\ZK}>1/2.
	\end{equation}
	If $g\in \BMO_\alpha(X)$ (perhaps non-measurable), then
	\begin{equation}\label{x10}
		\mu^*\{x\in B:\, |g(x)|>\lambda+\|g\|_{\BMO_\alpha}\}\le \varepsilon\cdot \mu^*\{x\in B:\, |g(x)|>\lambda\}
	\end{equation}
	for any ball $B\in \ZB$ and any number $\lambda>0$, provided
	\begin{equation}\label{x17}
		\mu^*\{x\in B:\, |g(x)|>\lambda\}\le \frac{\theta }{5\ZK\beta^2}\cdot \mu(B).
	\end{equation}
\end{proposition}
\begin{proof}
	Fix a ball $B$ and a number $\lambda>0$, satisfying  \e{x17}. Then consider the set $E(\lambda)=\{x\in B:\,|g(x)|>\lambda\}$, which can be non-measurable. Choose a measurable set $F$ such that
	\begin{equation}\label{x50}
		B\supset F\supset E(\lambda),\quad \mu(F)=\mu^*(E(\lambda))\le \frac{\theta }{5\ZK\beta^2}\cdot \mu(B) \,(\text{see } \e{x17}).
	\end{equation}
	Applying the density property (\lem{L12}) for any point $x\in F$
	we find a ball $G_0(x)\ni x$ such that $\mu(G_0(x)\cap F)\ge\alpha\theta \mu(G_0(x))/(2\beta)$. Then, applying the second part of \lem{L5}, one can find a sequence of balls $G_0=G_0(x)\subset G_1\subset G_2\subset\ldots $, satisfying the conditions of the lemma.  Let $n$ be the biggest integer such that \begin{equation}
		\mu(G_{n-1}(x)\cap F)\ge\frac{ \alpha\theta\cdot \mu(G_{n-1}(x))}{2\beta}.
	\end{equation}
	Clearly, the ball $G(x)=G_n$ satisfies
	\begin{equation}\label{x8}
		\frac{\alpha \theta}{2\beta^2}\le \frac{\mu(G(x)\cap F)}{\mu(G(x))}<\frac{\alpha\theta}{2\beta}\, (\text{see }\e{x71}).
	\end{equation}
	Since  $\mu(G_{n+1})\le \beta \mu(G_n^*)$ and $G^*(x)=G_n^*\subset G_{n+1}$ we obtain
	\begin{equation}\label{x73}
		\frac{\mu(G^*(x)\cap F)}{\mu(G^*(x))}\le  \frac{\beta\mu(G_{n+1}\cap F)}{ \mu(G_{n+1})}< \frac{\alpha\theta}{2}.
	\end{equation}
	The left hand side inequality in \e{x8} together with \e{x74} and \e{x50} imply
	\begin{equation}\label{x18}
		\mu(G^*(x))\le \ZK\mu(G(x))\le \frac{2\ZK\beta^2}{\alpha\theta}\cdot \mu(F)<  \mu(B).
	\end{equation}
	Applying \lem{L1-1}, then we find a sequence of pairwise disjoint balls $\{B_k\}\subset \{G(x):\, x\in F\}$ such that
	\begin{equation}\label{x11}
		F\subset \cup_kB_k^*.
	\end{equation}
	By the definition of $\alpha$-oscillation there are sets $E_k\subset B^*_k$ such that
	\begin{align}
		&\mu(E_k)\ge \alpha\mu(B^*_k),\label{x7}\\
		&\OSC_{E_k}(g)\le \OSC_{B_k^*,\alpha}(g)\le  \|g\|_{\BMO_\alpha}.\label{x12}
	\end{align}
	Applying the regularity property (see \df{D3}) and \e{x18}, we can write $\mu(B_k^*\cap B)\ge \theta \mu(B_k^*)$. Then from \e{x73} and \e{x7} we obtain
	\begin{equation}\label{x15}
		\mu(B_k^*\cap F)<\frac{\alpha\theta\mu(B_k^*)}{2}\le \frac{\alpha}{2} \mu(B_k^*\cap B)<  \frac{1}{2} \mu(B_k^*\cap B)
	\end{equation}
	and
	\begin{align}
		\mu(E_k\cap B)&\ge  \mu(B^*_k\cap B)-\mu(B_k^*\setminus E_k)\\
		&\ge \mu(B^*_k\cap B)-(1-\alpha)\mu(B_k^*)\\
		&\ge \mu(B_k^*\cap B)-\frac{1-\alpha}{\theta}\cdot \mu(B_k^*\cap B)\\
		&\ge  \frac{1}{2}  \mu(B_k^*\cap B),\label{x14}
	\end{align}
	respectively. Since both $B_k^*\cap F$ and $E_k\cap B$ are subsets of the set $B_k^*\cap B$, from \e{x14} and \e{x15} we obtain
	$\mu((E_k\cap B)\setminus F)=\mu(E_k\cap B)-\mu(B_k^*\cap F)>0$.
	Thus we can find  a point $x_k\in E_k\setminus F$ so that $|g(x_k)|\le\lambda$. Then by \e{x12} for any $x\in E_k$ we will have $|g(x)-g(x_k)|\le \|g\|_{\BMO_\alpha}$ and so $|g(x)|\le \lambda+\|g\|_{\BMO_\alpha}$. This implies 
	\begin{align}
		\mu^* (E\big(\lambda+\|g\|_{\BMO_\alpha})\big)&\le \sum_k\mu(B_k^*\setminus E_k)&\quad (\hbox{see }\e{x50},\e{x11})\\
		&\le (1-\alpha) \sum_k\mu(B_k^*)&\quad (\hbox{see }\e{x7})\\
		&\le \ZK(1-\alpha) \sum_k\mu(B_k)&\\
		&\le  \frac{2\beta^2 \ZK}{\alpha\theta}(1-\alpha) \sum_k\mu(B_k\cap F)&\quad (\hbox{see }\e{x8})\\
		&\le\varepsilon\cdot \mu(F)\\
		&= \varepsilon\cdot \mu^*(E(\lambda)).&\quad (\hbox{see }\e{x50})
	\end{align}
	Hence we get \e{x10}.
\end{proof}
\begin{proof}[Proof of \trm{T4}]
	Let $\alpha$ be as in \e{x74} and let $\varepsilon=1/2$. The inclusion $ \BLO(X)\subset \BLO_\alpha(X)$ immediately follows from the definition of $\BLO(X)$. For a function $f\in \BLO_\alpha(X)$ we have
	\begin{align}
		\mu^*\{x\in B:\,& f(x)-\INF_B(f)>\|f\|_{\BLO_\alpha}\}\\
		&\le \mu^*\{x\in B:\, f(x)-\INF_B(f)>\LOSC_{B,\alpha}(f)\}\\
		&\le (1-\alpha)\mu(B)=\frac{\theta}{8\ZK\beta^2 }\cdot \mu(B)\label{a40}
	\end{align}
and $\|f\|_{\BLO_\alpha}\ge \|f\|_{\BMO_\alpha}=\|f-\INF_B(f)\|_{\BMO_\alpha}$.	Thus, applying \pro {P}, for $g=f-\INF_B(f)$ and $\lambda>\|f\|_{\BLO_\alpha}$ we can write
	\begin{align}
		\mu^*\{x\in B:\,f(x)-\INF_B(f)>&\lambda+\|f\|_{\BLO_\alpha}\}\\
		&\le \frac{1}{2}\cdot \mu^*\{x\in B:\, f(x)-\INF_B(f)>\lambda\}.
	\end{align}
Applying this inequality consecutively, we get
	\begin{equation}
		\frac{1}{\mu(B)}\cdot \mu^*\{x\in B:\, f(x)-\INF_B(f)>\lambda\}\le c_1\exp \left(\frac{-c_2\lambda}{\|f\|_{\BLO_\alpha}}\right)
	\end{equation}
	with admissible constants $c_1$, $c_2$, which implies \e{x6}.
\end{proof}
\begin{proof}[Proof of \trm{T11}]
	The inclusion $ \BMO(X)\subset \BMO_\alpha(X)$ is immediate. To prove the converse embedding we let $\alpha$ be the number from \e{x74} and $\varepsilon=1/2$.  By the definition of $\alpha$-oscillation any ball $B$ contains a measurable set $E\subset B$ such that
	\begin{align}
		&\mu(E)\ge \alpha\mu(B),\\
		&\OSC_E(f)\le 2\OSC_{B,\alpha}(f).
	\end{align}
	Thus for $a_B=(\SUP_E(f)+\INF_E(f))/2$ we have
	\begin{align}
		\mu^*\{x\in B:\, &|f(x)-a_B|>\OSC_{B,\alpha}(f)\}\\
		&\le \mu^*\{x\in B:\, |f(x)-a_B|>\OSC_{E}(f)/2\}\\
		&\le\mu(B\setminus E) \le (1-\alpha)\mu(B),
	\end{align}
	and therefore
	\begin{align}
		\mu^*\{x\in B:\, &|f(x)-a_B|>\|f-a_B\|_{\BMO_\alpha}\}\\
		&=\mu^*\{x\in B:\, |f(x)-a_B|>\|f\|_{\BMO_\alpha}\}\\
		&\le  \mu^*\{x\in B:\, |f(x)-a_B|>\OSC_{B,\alpha}(f)\}\\
		&\le (1-\alpha)\mu(B)\\
		&=\frac{ \theta }{8\ZK\beta^2}\cdot \mu(B).
	\end{align}
	Then, applying \pro {P} consecutively, we conclude
	\begin{equation}\label{r53}
		\frac{1}{\mu(B)}\cdot \mu^*\{x\in B:\, |f(x)-a_B|>\lambda\}\le c_1\exp \left(\frac{-c_2\lambda}{\|f\|_{\BMO_\alpha}}\right).
	\end{equation}
	Note that the integral of a non-mesurable function is defined by the distribution function (see \e{r52}). Finally, from \e{r53} it follows that 
	\begin{equation*}
		\langle f\rangle_{\#, B}\le 2\left(\frac{1}{\mu(B)}\int_B|f-a_B|^r\right)^{1/r}\lesssim \|f\|_{\BMO_\alpha}
	\end{equation*}
	for any ball $B$, and so we get $\|f\|_\BMO\lesssim 	\|f\|_{\BMO_\alpha}$. This completes the proof of \trm{T11}.
\end{proof}

\section{Examples}\label{S6}

\subsection{Maximal operators on $\ZR^d$}
Let $\xi:\ZR^d\to \ZR$ be a non-negative spherical function, which is decreasing with respect to $|x|$ and satisfies
\begin{equation}\label{x60}
	J(\xi)=\int_{\ZR^d}\xi(x)\log(2+|x|)dx<\infty.
\end{equation}
Let $\phi=\{\phi_r(x,y),\,r>0\}$ be a family of kernels, satisfying the bound
\begin{align}
	&\int_{\ZR^d}\phi_r(x,y)dy=1,\\
	&\frac{c\cdot \ZI_{\{|y|\le r\}}(y)}{r^d}\le \phi_r(x,y)\le \frac{1}{r^d}\cdot \xi\left(\frac{y}{r}\right),\label{r51}
\end{align}
where $c$ is  positive constant. Consider the maximal functions
	\begin{align}
		&\MM_\phi f(x)=\sup_{r>0}\int_{\ZR^d}|f(t)|\phi_r(x,y)dy,\label{x75}\\
		&\bar\MM_\phi f(x)=\sup_{r>0,\, |x'-x|< r}\int_{\ZR^d}|f(t)|\phi_r(x',y)dy.\label{x76}
	\end{align}
Observe that these operators are more general versions of operators \e{x55} and \e{r55}. For both operators we have 
\begin{corollary}
Let a function $\xi:\ZR^d\to \ZR$ be as above and suppose that the kernels $\phi_r(x,y),\,r>0,$ satisfy \e{r51}. If $f\in \BMO(\ZR^d)$ and $\MM_\phi(f)$ is not identically infinite, then $\MM_\ZG(f)\in \BLO(\ZR^d)$ and 
	\begin{equation}
		\|\MM_{\phi}(f)\|_{\BLO(\ZR^d)}\lesssim I(\omega)\|f\|_{\BMO(\ZR^d)}.
	\end{equation}
The same bound also holds for the operator $\bar\MM_\phi$. 
\end{corollary}
\begin{proof}
	This immediately follows from \trm{KC2}. We just need to indicate the corresponding KB-couples generating these operator. Let $B=B(x,r)\subset \ZR^d$ be the open ball of the center $x$ and the radius $r$. Setting $\phi_B(y)=\phi_r(x,y)$ we will have \e{x45}. Then one can check that the family of functions $\{\phi_B\}$ form a kernel-structure with a modulus of continuity, satisfying \e{y7}. Moreover, condition \e{y4} obviously holds with a modulus of continuity $\omega(t)\sim \frac{\xi(0)}{\xi(1)}\cdot \xi(t^{1/d})$, $t\ge 1$ and it immediately follows from \e{r51}. On the other hand, since $\xi(x)$ is a spherical function, passing to the spherical coordinates, we get
\begin{align*}
	\infty> J(\xi)&\ge\int_{|x|\ge 1}\xi(x)\log(2+|x|)dx=c_d\int_{r\ge 1}\xi(r)\log(2+r)r^{d-1}dr\\
	&=\frac{c_d}{d}\int_1^\infty\xi(t^{1/d})\log(2+t^{1/d})dt\gtrsim \int_1^\infty\omega(t)\log(2+t)dt=I(\omega).
\end{align*}
Hence we have $I(\omega)<\infty$ that gives \e{y7}. Having suitable kernel-structure $\{\phi_B\}$, let us indicate the ball-structures for both maximal functions $\MM_\phi$ and $\bar \MM_\phi$.
For $\MM_\phi$ we define $\ZB(x)$ to be the family of balls centered at $x$. On can check 
\begin{equation}
	\MM_\phi f(x)=\sup_{B\in \ZB(x)}\langle f\rangle_B.
\end{equation}
Choosing $\ZB(x)$ to be the set of all balls containing the point $x$, the same we can write for the operator $\bar \MM_\phi$. It remains to add that in both cases the families $\ZB(x)$ satisfy conditions of G1) and G2) of \df{D4}.
\end{proof}
We can also consider the $2\pi$-periodic version of the operators \e{x75} and \e{x76}. Namely, our function $\xi$ is the same as in \e{x60}, but the kernels \begin{equation}
	\phi_n(x,y)=	\phi_n(x_1,\ldots,x_d,y_1,\ldots,y_d)
\end{equation} 
are defined on the $d$-dimensional torus $\ZT^d$, where $\ZT=\ZR/(2\pi\ZZ)$. Equivalently, those are functions that are $2\pi$-periodic with respect to each variables $x_j$ and $y_j$. We require the following conditions:
\begin{align}
	&\int_{\ZT^d}\phi_r(x,y)dy=1,\\
	&\frac{c\cdot \ZI_{\{|y|\le r\}}(y)}{r^d}\le \phi_r(x,y)\le \frac{1}{r^d}\cdot \xi\left(\frac{y}{r}\right),
\end{align}
provided $ x_j,y_j\in (-\pi,\pi)$. Using the same argument we obtain that the analogous maximal operators are bounded on $\BMO(\ZT^d)$. As a corollary of this we have.
\begin{corollary}
	Let $\sigma_n(x,f)$ be the operator of $(C,\alpha)$, $\alpha>0$, Ces\`aro means of the Fourier series of the function $f\in L^1(\ZT)$. Then the maximal operator 
	\begin{equation}
		\sigma^*(x,f)=\sup_n|\sigma_n(x,f)|
	\end{equation}
boundedly maps $\BMO(\ZT^d)$ into itself.
\end{corollary}
\subsection{Dyadic maximal operators}
The next example will be considered in the Lebesgue measure space on $[0,1]$, where the ball-basis $\ZB$ is the set of dyadic intervals
\begin{equation*}
	\left[\frac{j-1}{2^k},\frac{j}{2^k}\right),\quad 1\le j\le 2^k,\quad n\ge 0.
\end{equation*}
 Let $I_n(x)$ be the unique dyadic interval of length $2^{-n}$, $n\ge 0$ , containing the point $x\in [0,1)$. Denote by $\ZB(x)$ the sequence of dyadic intervals $I_n(x)$, $n=0,1,2\ldots $. Consider a numerical sequence $\alpha=\{\alpha_k\ge 0:\, k\ge 0\}$, satisfying 
	\begin{align}\label{x58}
		J(\alpha)=\sum_{k=0}^\infty (k+1)\alpha_k<\infty.
	\end{align}
Let $I$ be a dyadic interval of length $2^{-n}$, $n\ge 0$. Clearly, we have $I=I_n(x)$ for any $x\in I$.  We associate with this interval the kernel function 
\begin{equation}
	\phi_I(y)=\phi_n(x,y)=\sum_{k=0}^n\alpha_{n-k}\frac{\ZI_{I_k(x)}(y)}{|I_k(x)|},\quad x,y\in [0,1),\quad n=0,1,2,\ldots,
\end{equation}
and consider the corresponding maximal function
\begin{equation}\label{x59}
	\MM_{\alpha}f(x)=\sup_{n\ge 0} \int_\ZR |f(t)|\phi_n(x,t)dt.
\end{equation}
\begin{corollary}
	Under the condition \e{x58} maximal operator \e{x59} maps boundedly dyadic $\BMO(0,1)$ into dyadic $\BLO(0,1)$.
\end{corollary}
\begin{proof}
 One can check that the collections $\ZB(x)$ together with the family of kernels $\{\phi_I:\, I\in \ZB\}$ form a KB-couple. Then we define 
\begin{equation*}
	\omega(t)=\sum_{k=0}^\infty\left(\sum_{j\ge k}\frac{\alpha_j}{2^j}\right)\ZI_{[2^k,2^{k+1})}(t), \quad t\ge 0.
\end{equation*}
Clearly this is a modulus of continuity. Besides, we have
\begin{align}
	I(\omega)&=1+\int_1^\infty\omega(t)\log(2+t)dt\\
	&\lesssim 1+\sum_{k=0}^\infty (k+1)2^k\sum_{j\ge k}\frac{\alpha_j}{2^j}\\
	&=1+\sum_{j=0}^\infty\frac{\alpha_j}{2^j}\sum_{k=0}^j(k+1)2^k\le 1+\sum_{j=0}^\infty(j+1)\alpha_j=1+J(\alpha)< \infty.
\end{align}
Then, let us check the function $\phi_I$ satisfies \e{y4}. The left hand side bound in \e{y4} is clear. It remain to show for $I=I_n(x)$ that 
\begin{equation}\label{r57}
	\phi_I(y)=\sum_{j= 0}^n\alpha_{n-j} \frac{\ZI_{I_j(x)}(y)}{|I_j(x)|}\le \frac{1}{|I|}\cdot \omega\left(\frac{d(y,I)}{|I|}\right),\quad y\in [0,1).
\end{equation}
First, suppose that $y\not \in I=I_n(x)$. Then we will have $y\in I_k(x)\setminus I_{k+1}(x)$ for some $ 0\le k\le n-1$ and therefore
\begin{equation*}
 \phi_I(y)=\sum_{j=0}^k\alpha_{n-j} \cdot 2^{j}.
\end{equation*} 
We also have $d(y,I)=2^{n-k}|I|$. Hence the right hand side of \e{r57} is equal
\begin{align*}
	\frac{1}{|I|}\cdot \omega\left(2^{n-k}\right)=2^n\sum_{j\ge n-k}\frac{\alpha_j}{2^j}\ge \sum_{j= n-k}^{n}\alpha_j\cdot 2^{n-j}=\sum_{j= 0}^{k}\alpha_{n-j}\cdot 2^{j}.
\end{align*}
Thus \e{r57} follows if $y\not\in I$. In the case $x\in I$ \e{r57} holds trivially. 
Thus we conclude that the collection of kernels $\{\phi_I\}$ satisfy the hypothesis of \trm{T2} and \cor {KC2}. This completes the proof of corollary.  
\end{proof}
\begin{remark}
	One can similarly define a maximal operator based on general doubling martingale-filtration instead of the dyadic intervals. 
\end{remark}

\end{document}